\documentclass[11pt]{amsart}
\usepackage{amssymb,amsmath,amstext}

\theoremstyle{plain}

\newtheorem{theorem}{Theorem}[section]

\newtheorem{proposition}[theorem]{Proposition}

\newtheorem{observation}[theorem]{Observation}
\newtheorem{corollary}[theorem]{Corollary}
\newtheorem{claim}[theorem]{Claim}

\theoremstyle{definition}

\theoremstyle{remark}
\newtheorem{example}[theorem]{Example}

\begin{document}

\newcommand{\Dkk}{\Delta_k \times \Delta_k}
\newcommand{\aij}{A_{ij}}
\newcommand{\cf}{\mathcal{F}}
\newcommand{\F}{\mathcal{F}}
\newcommand{\I}{\mathcal{I}}
\newcommand{\J}{\mathcal{J}}
\newcommand{\M}{\mathcal{M}}

\newcommand{\T}{\mathcal{T}}

\newcommand{\e}{\varepsilon}
\newcommand{\vj}{\vec{j}}
\newcommand{\vw}{\vec{w}}
\newcommand{\vz}{\vec{z}}
\newcommand{\vh}{\vec{h}}
\newcommand{\vx}{\vec{x}}
\newcommand{\vy}{\vec{y}}
\newcommand{\vk}{\vec{k}}
\newcommand{\vv}{\vec{v}}
\newcommand{\vl}{\vec{\ell}}
\newcommand{\cl}{\mathcal L}
\newcommand{\cm}{\mathcal M}
\newcommand{\cn}{\mathcal N}
\newcommand{\dist}{\text{dist}}

\title[Edge-covers in $d$-interval hypergraphs]{Edge-covers in $d$-interval hypergraphs}

\author{Ron Aharoni}\thanks{The research of the first  author was
supported by  BSF grant no.
2006099, by an ISF grant and by the Discount Bank
Chair at the Technion.}
\address{Department of Mathematics,
Technion\\
Haifa, Israel} \email{raharoni@gmail.com}

\author{Ron Holzman}\thanks{}
\address{Department of Mathematics,
Technion\\
Haifa, Israel} \email{holzman@tx.technion.ac.il}
\author{Shira Zerbib}\thanks{The research of the third author was partly supported by the New-England Fund, Technion.}

\address{Department of Mathematics,
University of Michigan, Ann Arbor, USA} \email{zerbib@umich.edu}


\begin{abstract}

  A $d$-interval hypergraph has $d$ disjoint copies of the unit interval as its vertex set, and each edge is the union of $d$ subintervals, one on each copy. Extending a classical result of Gallai on the case
$d=1$, Tardos and Kaiser used
topological tools to bound the ratio between the transversal
  number and
  the matching number in such hypergraphs.
  We take a dual point of view, and bound the edge-covering number (namely the minimal number of edges covering the entire vertex set) in terms of a parameter expressing independence of systems of partitions of the $d$ unit intervals.
   The main tool we use is an extension of  the KKM theorem to products of simplices, due to Peleg. Our approach also yields a new proof of the Tardos-Kaiser result.
  \end{abstract}

\maketitle


\section{Introduction}
\label{sec.int}

A {\em $d$-interval hypergraph} has vertex set $V=\bigcup_{i=1}^d U^i$, where the $U^i$ are $d$ disjoint copies of the unit interval $[0,1]$. It has a finite edge set $E$, each edge being of the form $e=\bigcup_{i=1}^d I^i$ where $I^i$ is a non-empty closed subinterval of $U^i$. In the case $d=1$ we get the familiar interval hypergraphs. There are other variants of $d$-interval hypergraphs in the literature \cite{gl2, kaiser, alon, berger, AKZ}, including those where edges consist of unions of $d$ intervals on the same copy of $[0,1]$, and discrete versions where the vertex set is a finite ordered set; these will not be discussed here.

We recall that in a hypergraph $H=(V,E)$, a {\em matching} is a set of disjoint edges, and $\nu(H)$ denotes the {\em matching number} -- the maximal size of a matching. A {\em transversal} (also called {\em vertex-cover}) is a set
of vertices meeting all edges, and $\tau(H)$ denotes the \emph{transversal number} -- the minimal size of a transversal. We always have $\nu(H) \le \tau(H)$.

By reversing the roles of vertices and edges (duality of hypergraphs), we get two corresponding concepts. Namely, in a hypergraph $H=(V,E)$, a set of vertices is {\em strongly independent} if no two of them belong to an edge. (An independent set is one that does not contain any edge, hence the ``strongly" in our terminology.) We write $\iota(H)$ for the {\em strong independence number} -- the maximal size of a strongly independent set ($\iota(H) = \infty$ is allowed). A set of edges is an {\em edge-cover} if their union is $V$. We write $\rho(H)$ for the {\em edge-covering number} -- the minimal size of an edge-cover ($\rho(H) = \infty$ if $\cup E \ne V$). We always have $\iota(H) \le \rho(H)$.

It is well known that for interval hypergraphs equality holds for each of these pairs of parameters (part (a) below is due to Gallai \cite{gal62}):

\begin{theorem}\label{gallai}
Let $H=(V,E)$ be an interval hypergraph. Then:
\begin{itemize}
\item[(a)] $\nu(H)=\tau(H)$.
\item[(b)] $\iota(H)=\rho(H)$.
\end{itemize}
\end{theorem}

\begin{proof} We briefly recall the standard constructive arguments.
\begin{itemize}
\item[(a)] Let $I=[a,b]$ be an interval in $E$ with the leftmost right endpoint. Then we can place $I$ in the matching being constructed, $b$ in the transversal, and proceed by induction with the subhypergraph consisting of those intervals in $E$ having their left endpoint to the right of $b$.
\item[(b)] We may assume that $\cup E = V$, otherwise $\iota(H)=\rho(H)=\infty$. Let $I=[0,c]$ be an interval in $E$ with the rightmost right endpoint among those starting at $0$. Then we can place $0$ in the strongly independent set being constructed, $I$ in the edge-cover, and proceed by induction with the hypergraph having as ground set $[c+\varepsilon,1]$ for small enough $\varepsilon >0$, and as edges the non-empty intersections of intervals in $E$ with the ground set.
\end{itemize}
\end{proof}

For $d$-interval hypergraphs with $d \ge 2$ these equalities need not hold, but one can bound the ratio between the corresponding pairs of parameters. In the case of $\nu$ and $\tau$ this was a challenging problem, solved by Tardos \cite{tardos} for $d=2$ and Kaiser \cite{kaiser} for general $d$:

\begin{theorem}[\cite{tardos, kaiser}]\label{t:kaiser}
Let $H$ be  a $d$-interval hypergraph, $d \ge 2$. Then $\tau(H) \le d(d-1)\nu(H)$.
\end{theorem}

Both Tardos and Kaiser used topological methods. Alon \cite{alon} obtained a slightly weaker bound by combinatorial means.

In the case of $\iota$ and $\rho$, bounding the ratio for $d$-interval hypergraphs is easy:

\begin{observation}\label{rhoiota}
Let $H$ be a $d$-interval hypergraph. Then $\rho(H) \le d \iota(H)$, and this bound is tight.
\end{observation}

\begin{proof} For $i=1,\ldots,d$, let $H^i$ be the interval hypergraph on $U^i$ whose edges are the intersections of the edges of $H$ with $U^i$. Clearly $\iota(H^i) \le \iota(H)$, and so by Theorem \ref{gallai}(b) we have $\rho(H^i) \le \iota(H)$, $i=1,\ldots,d$. Together, these $d$ edge-covers yield $\rho(H) \le d \iota(H)$.

To see that this is tight, fix $d$ and $n$. We construct a $d$-interval hypergraph $H=(V,E)$ with $\iota(H)=n$ and $\rho(H)=dn$. For each $i=1,\ldots,d$, we use two systems of subintervals of $U^i$. The first consists of $J^i_j=[\frac{j-1}{n},\frac{j}{n}]$, $j=1,\ldots,n$; the second consists of $K^i_k=[\frac{k-1}{(dn)^2},\frac{k}{(dn)^2}]$, $k=1,\ldots,(dn)^2$. An edge $e \in E$ is determined by a choice of $i \in \{1,\ldots,d\}$, $j \in \{1,\ldots,n\}$, and $k_{i'} \in \{1,\ldots,(dn)^2\}$ for each $i' \in \{1,\ldots,d\} \setminus \{i\}$. The corresponding edge is $e=J^i_j \cup \bigcup_{i' \ne i} K^{i'}_{k_{i'}}$. Clearly, no strongly independent set can contain points in two distinct $U^i, U^{i'}$, and it can contain at most $n$ points in the same $U^i$; hence $\iota(H)=n$. The total length of an edge $e \in E$ is $\frac{1}{n} + \frac{d-1}{(dn)^2} = \frac{d^2n + d - 1}{(dn)^2}$, and hence the total length of $dn-1$ such edges is $\frac{(dn-1)(d^2n+d-1)}{(dn)^2} < d$. Thus, there is no edge-cover of size $dn-1$, and hence $\rho(H)=dn$.
\end{proof}

There is, however, another way of looking at Theorem \ref{gallai}, which will lead to a more interesting extension of part (b) to the case of $d$-interval hypergraphs.

A system of $n$ subintervals of $[0,1]$ of the form $$[0,c_1), (c_1,c_2), \ldots, (c_{n-2},c_{n-1}), (c_{n-1},1],$$ where $0 \le c_1 \le c_2 \le \cdots \le c_{n-2} \le c_{n-1} \le 1$, will be called an {\em $n$-partition} of $[0,1]$. The subintervals will be referred to as the first, the second, ..., the $(n-1)$-th, the $n$-th {\em cell} of the partition. Note that the $j$-th cell is empty if $c_{j-1}=c_j$ (with $c_0=0, c_n=1$).

Theorem \ref{gallai} can be equivalently re-stated as:

\begin{theorem}\label{galpart}
Let $H=(V,E)$ be an interval hypergraph, and let $n$ be a positive integer. Then:
\begin{itemize}
\item[(a)] If every $n$-partition of $[0,1]$ has a cell that contains an edge of $H$, then there exists an $n$-partition all of whose cells contain edges of $H$.
\item[(b)] If every $n$-partition of $[0,1]$ has a cell that is contained in an edge of $H$, then there exists an $n$-partition all of whose cells are contained in edges of $H$.
\end{itemize}
\end{theorem}

Now we generalize the partition concept to the setting where $V= \bigcup_{i=1}^d U^i$. A {\em $d \times n$-partition} of $V$ consists of $d$ $n$-partitions of $U^1,\ldots,U^d$ respectively. Its total number of cells is $dn$. Any union of $d$ cells, one from each $U^i$, will be called a {\em $d$-cell}. Thus, there are $n^d$ $d$-cells.

Part (a) of Theorem \ref{galpart} has the following $d$-counterpart:

\begin{theorem}\label{disjointcells}
Let $H=(V,E)$ be a $d$-interval hypergraph, $d \ge 2$, and let $n$ be a positive integer. If every $d \times n$-partition of $V$ has a $d$-cell that contains an edge of $H$, then there exists a $d \times n$-partition having at least $\frac{n}{d-1}$ disjoint $d$-cells that contain edges of $H$.
\end{theorem}

Essentially, Theorem \ref{t:kaiser} was proved in \cite{tardos, kaiser} by establishing Theorem \ref{disjointcells}, and then applying it with the largest $n$ such that $d(n-1) < \tau(H)$.

Our main result is the following $d$-counterpart of Theorem \ref{galpart}(b):

\begin{theorem}\label{maintheorem}
 Let $H=(V,E)$ be  a $d$-interval hypergraph, and let $n$ be a positive integer. Assume that every $d \times n$-partition of $V$ has a $d$-cell that is contained in an edge of $H$. Then there exists a $d \times n$-partition all of whose cells are contained in edges of $H_0$, where $H_0$ is a subhypergraph of $H$ of size at most $(1+\ln d)n$. In particular, $\rho(H) \le (1+\ln d)n$. \\Moreover, for $d=2$ the upper bound on the size of $H_0$ can be improved to $n$, and hence $\rho(H) \le n$ for $d=2$.
\end{theorem}

Although Theorem \ref{disjointcells} is not new, we give here a new proof of it, which goes along similar lines to our proof of Theorem \ref{maintheorem}. Both proofs are topological, using extensions due to Peleg \cite{peleg, peleg2} of the well-known KKM theorem \cite{kkm}.

In Section \ref{ps} we recall the KKM theorem about coverings of a simplex, and a dual variant due to Sperner \cite{sperner}. Then, as a warm-up to our main proofs, we apply them to get topological proofs of both parts of Theorem \ref{galpart} (and thus of Theorem \ref{gallai}). Then, in Section \ref{tools}, we recall the extensions due to Peleg of the KKM/Sperner theorems to coverings of a product of simplices. We also recall some well-known bounds on the ratio between integer and fractional versions of the matching and transversal numbers of hypergraphs. Equipped with these tools, we prove Theorems \ref{disjointcells} and \ref{maintheorem} in Section \ref{main}. In the final Section \ref{examples}, we discuss the tightness of the bound in Theorem \ref{maintheorem}. We show that $\rho(H) \le n$ is best possible for $d=2$, and that it does not hold anymore for $d>2$. The $(1+\ln d)n$ upper bound for general $d$ can be somewhat improved, but we do not know by how much.

\section{Topological proofs for interval hypergraphs}\label{ps}

We denote by $\Delta_{n-1}$ the standard simplex in $\mathbb{R}^n$, namely
\[ \Delta_{n-1} = \{\vec{x}=(x_1,\ldots,x_n) \mid \sum_{j=1}^n x_j =1, x_j \ge 0, j=1,\ldots,n\}. \]
For $S\subseteq [n]=\{1,\ldots,n\}$, $S \ne \emptyset$, we denote by $F(S)$ the face of
$\Delta_{n-1}$ spanned by the corresponding unit vectors, i.e.,
\[ F(S) = \{\vec{x}=(x_1,\ldots,x_n) \in \Delta_{n-1} \mid x_j=0 \textrm{ for all }j \not \in S\}. \]

\begin{theorem}[KKM \cite{kkm}]\label{kkm}
Let $A_1, \ldots ,A_n$ be subsets of $\Delta_{n-1}$ that are all closed or all open. Suppose that for every $\emptyset \ne S \subseteq [n]$ we have $F(S)\subseteq \bigcup_{j \in S}A_j$. Then $\bigcap_{j=1}^n A_j \neq \emptyset$.
\end{theorem}

We shall use the following immediate corollary of the KKM theorem:

\begin{corollary}\label{kkmcor}
Let $A_1,\ldots,A_n$ be subsets of $\Delta_{n-1}$ that are all closed or all open. Suppose that
\begin{itemize}
\item[(a)] $\bigcup_{j=1}^n A_j = \Delta_{n-1}$, and
\item[(b)] for all $\vx =(x_1,\ldots,x_n) \in \Delta_{n-1}$ and all $j \in [n]$, $x_j=0 \Rightarrow \vx \notin A_j$.
\end{itemize}
Then $\bigcap_{j=1}^n A_j \ne \emptyset$.
\end{corollary}

There is a dual variant which differs from Corollary \ref{kkmcor} by reversing the conclusion of $x_j=0$:

\begin{theorem}[Sperner \cite{sperner}]\label{sperner}
Let $B_1, \ldots ,B_n$ be subsets of $\Delta_{n-1}$ that are all closed or all open. Suppose that
\begin{itemize}
\item[(a)] $\bigcup_{j=1}^n B_j = \Delta_{n-1}$, and
\item[(b)] for all $\vx =(x_1,\ldots,x_n) \in \Delta_{n-1}$ and all $j \in [n]$, $x_j=0 \Rightarrow \vx \in B_j$.
\end{itemize}
Then $\bigcap_{j=1}^n B_j \ne \emptyset$.
\end{theorem}

It is possible to derive Theorem \ref{sperner} from Corollary \ref{kkmcor} by letting $A_j = \Delta_{n-1} \setminus B_j$.

There is a canonical bijection between $n$-partitions of $[0,1]$, as defined in the Introduction, and points in $\Delta_{n-1}$. Namely, the $n$-partition $$[0,c_1), (c_1,c_2), \ldots, (c_{n-2},c_{n-1}), (c_{n-1},1]$$ corresponds to the point $\vx =(x_1,x_2,\ldots,x_{n-1},x_n)$, where $x_j = c_j - c_{j-1}$ is the length of the $j$-th cell (with $c_0=0,c_n=1$). In the following, we identify every $n$-partition of $[0,1]$ with the corresponding $\vx \in \Delta_{n-1}$, and use the topological results above to prove Theorem \ref{galpart}.

\medskip

{\em Proof of Theorem \ref{galpart}(a) from Corollary \ref{kkmcor}:} Given the interval hypergraph $H=(V,E)$ and the positive integer $n$, we define $A_j \subseteq \Delta_{n-1}$ for $j \in [n]$ by:
\[ \vx \in A_j \,\,\Leftrightarrow \textrm{ the }j\textrm{-th cell of the }n\textrm{-partition }\vx\textrm{ contains an edge of }H \]
It is easy to check that the sets $A_1,\ldots,A_n$ are open, and that the premise of Theorem \ref{galpart}(a) renders them a covering of $\Delta_{n-1}$. Moreover, condition (b) of Corollary \ref{kkmcor} holds, since $x_j=0$ means that the $j$-th cell is empty. It follows that there exists $\vx \in \bigcap_{j=1}^n A_j$. Then, all cells of the $n$-partition corresponding to $\vx$ contain edges of $H$, as desired. \hspace{\stretch{1}}$\square$

\medskip

{\em Proof of Theorem \ref{galpart}(b) from Theorem \ref{sperner}:} Given the interval hypergraph $H=(V,E)$ and the positive integer $n$, we define $B_j \subseteq \Delta_{n-1}$ for $j \in [n]$ by:
\[ \vx \in B_j \,\,\Leftrightarrow \textrm{ the }j\textrm{-th cell of the }n\textrm{-partition }\vx\textrm{ is contained in an edge of }H \]
It is easy to check that the sets $B_1,\ldots,B_n$ are closed, and that the premise of Theorem \ref{galpart}(b) renders them a covering of $\Delta_{n-1}$. Moreover, condition (b) of Theorem \ref{sperner} holds, since an empty cell is contained in any edge. It follows that there exists $\vx \in \bigcap_{j=1}^n B_j$, meaning that all cells of the corresponding $n$-partition are contained in edges of $H$. \hspace{\stretch{1}}$\square$

\section{Tools for handling $d$-interval hypergraphs}\label{tools}

We denote by $P_{d,n}$ the Cartesian product of $d$ copies of $\Delta_{n-1}$, i.e.,
\[ P_{d,n} = \Delta_{n-1}\times \Delta_{n-1} \times \cdots \times \Delta_{n-1}=(\Delta_{n-1})^d. \]
For $i \in [d]$ and $\emptyset \ne S\subseteq [n]$ we denote by $P_{d,n}(i,S)$ the polytope obtained from $P_{d,n}$ by replacing the $i$-th factor in the product by its face $F(S)$, namely
$$P_{d,n}(i,S)=\Delta_{n-1}\times \cdots \times F(S) \times \cdots \times \Delta_{n-1}.$$

\begin{theorem}[Peleg \cite{peleg}]\label{peleg}
Let $A^i_j$, $i=1,\ldots,d$, $j=1,\ldots,n$, be closed subsets of $P_{d,n}$. Suppose that for every $i \in [d]$ and $\emptyset \ne S\subseteq [n]$ we have $P_{d,n}(i,S)\subseteq \bigcup_{j \in S} A^i_j.$
Then
$\bigcap_{i=1}^d \bigcap_{j=1}^n A^i_j \neq \emptyset$.
\end{theorem}

This is a generalization of the KKM theorem to $d$-fold products of simplices, and it likewise has the following immediate corollary:

\begin{corollary}\label{pelegcor}
Let $A^i_j$, $i=1,\ldots,d$, $j=1,\ldots,n$, be closed subsets of $P_{d,n}$. Suppose that
\begin{itemize}
\item[(a)] for all $i \in [d]$, $\bigcup_{j=1}^n A^i_j = P_{d,n}$, and
\item[(b)] for all $\vx =(x^1_1,\ldots,x^1_n;\cdots;x^d_1,\ldots,x^d_n) \in P_{d,n}$, all $i \in [d]$ and all $j \in [n]$, $x^i_j=0 \Rightarrow \vx \notin A^i_j$.
\end{itemize}
Then $\bigcap_{i=1}^d \bigcap_{j=1}^n A^i_j \neq \emptyset$.
\end{corollary}

Here, too, there is a dual variant which differs from Corollary \ref{pelegcor} by reversing the conclusion of $x^i_j=0$:

\begin{theorem}[Peleg \cite{peleg2}]\label{dualpeleg}
Let $B^i_j$, $i=1,\ldots,d$, $j=1,\ldots,n$, be closed subsets of $P_{d,n}$. Suppose that
\begin{itemize}
\item[(a)] for all $i \in [d]$, $\bigcup_{j=1}^n B^i_j = P_{d,n}$, and
\item[(b)] for all $\vx =(x^1_1,\ldots,x^1_n;\cdots;x^d_1,\ldots,x^d_n) \in P_{d,n}$, all $i \in [d]$ and all $j \in [n]$, $x^i_j=0 \Rightarrow \vx \in B^i_j$.
\end{itemize}
Then $\bigcap_{i=1}^d \bigcap_{j=1}^n B^i_j \neq \emptyset$.

\end{theorem}

We remark that, unlike the case $d=1$, there seems to be no direct way to deduce Theorem \ref{dualpeleg} from Corollary \ref{pelegcor}. Another comment is that Peleg's theorems hold also for products of simplices of different dimensions, but were stated above in the special case (needed here) where the simplices have the same dimension.

\medskip

The other known results that we need concern hypergraphs $H=(V,E)$ on finite vertex sets, and involve fractional versions of our hypergraph invariants. We recall that a {\em fractional matching} is a function $f: E \to \mathbb{R}_+$ such that $\sum_{e \ni v} f(e) \le 1$ for all $v \in V$. The {\em fractional matching number} $\nu^*(H)$ is the maximum of $\sum_{e \in E} f(e)$ over all fractional matchings in $H$. Similarly, a {\em fractional transversal} is a function $g: V \to \mathbb{R}_+$ such that $\sum_{v \in e} g(v) \ge 1$ for all $e \in E$. The {\em fractional transversal number} $\tau^*(H)$ is the minimum of $\sum_{v \in V} g(v)$ over all fractional transversals in $H$. By linear programming duality we always have $\nu^*(H)=\tau^*(H)$.

Analogously, a {\em fractional strongly independent set} is a function $f: V \to \mathbb{R}_+$ such that $\sum_{v \in e} f(v) \le 1$ for all $e \in E$. The {\em fractional strong independence number} $\iota^*(H)$ is the maximum of $\sum_{v \in V} f(v)$ over all fractional strongly independent sets in $H$. A {\em fractional edge-cover} is a function $g: E \to \mathbb{R}_+$ such that $\sum_{e \ni v} g(e) \ge 1$ for all $v \in V$. The {\em fractional edge-covering number} $\rho^*(H)$ is the minimum of $\sum_{e \in E} g(e)$ over all fractional edge-covers in $H$. We always have $\iota^*(H)=\rho^*(H)$; both are defined as $\infty$ if $\cup E \ne V$.

We further recall that a hypergraph $H=(V,E)$ is {\em $d$-partite} if there exists a partition $V^1,\ldots,V^d$ of $V$ such that $|e \cap V^i|=1$ for every $e \in E$ and $i \in [d]$. We will use the following bound on the ratio $\nu^*/\nu$:

\begin{theorem}[Furedi \cite{furedi}]\label{furedi}
Let $H$ be a $d$-partite hypergraph, $d \ge 2$. Then $\nu(H) \ge \frac{\nu^*(H)}{d-1}$.
\end{theorem}

We recall that the degree of a vertex $v$ is the number of edges $e$ that contain $v$. We denote by $\Delta(H)$ the {\em maximal degree} in $H$. Similarly, the size of an edge $e$ is the number of vertices $v$ contained in $e$. We denote by $r(H)$ the {\em rank} of $H$, i.e., the maximal size of an edge. A well-known bound on the ratio $\tau/\tau^*$ is:

\begin{theorem}[Lov\'asz \cite{lovasz}]\label{lovasz}
Let $H$ be  a hypergraph with $\Delta(H)=d \ge 1$. Then $\tau(H) \le (1+\ln
d)\tau^*(H)$.
\end{theorem}

We will use this bound with the roles of vertices and edges reversed, namely:

\begin{corollary}\label{lovaszcor}
Let $H$ be a hypergraph with $r(H)=d \ge 1$. Then $\rho(H) \le (1+\ln d)\rho^*(H)$.
\end{corollary}

In the case $d=2$, when $H$ is bipartite, the situation is even better:

\begin{theorem}[Gallai \cite{gal59}]\label{bipgallai}
Let $G$ be a bipartite graph. Then $\rho(G)=\rho^*(G)$.
\end{theorem}

\section{Proof of Theorems \ref{disjointcells} and \ref{maintheorem}}\label{main}

Extending what we did in Section \ref{ps} in a natural way, we now identify every $d \times n$-partition of $V= \bigcup_{i=1}^d U^i$ with a point $\vx =(x^1_1,\ldots,x^1_n;\cdots;x^d_1,\ldots,x^d_n)$ in $P_{d,n}=(\Delta_{n-1})^d$. Under this identification, $x^i_j$ is the length of the $j$-th cell in the $n$-partition of $U^i$.

We use the set $[n]^d$ to index the $d$-cells of a $d \times n$-partition of $V$. That is, a $d$-tuple $\vj =(j_1,\ldots,j_d) \in [n]^d$ will index the $d$-cell formed by the $j_1$-th cell in the $n$-partition of $U^1$, ..., the $j_d$-th cell in the $n$-partition of $U^d$. We refer to this as the $\vj$-th $d$-cell.

\medskip

{\em Proof of Theorem \ref{disjointcells}:} Given the $d$-interval hypergraph $H=(V,E)$, $d \ge 2$, and the positive integer $n$, we define $Y_{\vj} \subseteq P_{d,n}$ for $\vj \in [n]^d$ by:
\[ \vx \in Y_{\vj} \,\,\Leftrightarrow \textrm{ the }\vj\textrm{-th }d\textrm{-cell of the }d \times n\textrm{-partition }\vx\textrm{ contains an edge of }H \]
The premise of the theorem says that the sets $Y_{\vj}$, $\vj \in [n]^d$, cover $P_{d,n}$. It is easy to verify that the sets $Y_{\vj}$ are open. Moreover, if any of the cells in a $d$-cell is empty then this $d$-cell cannot contain an edge of $H$. Therefore, for $\vx \in P_{d,n}$, $i \in [d]$, and $\vj \in [n]^d$ we have the implication $x^i_{j_i}=0 \Rightarrow \vx \notin Y_{\vj}$.

We now associate with every $\vx \in P_{d,n}$ a $d$-dimensional $n \times \cdots \times n$ array of nonnegative real numbers $S(\vx)=(s_{\vj}(\vx))_{\vj \in [n]^d}$, defined by:
\[ s_{\vj}(\vx)=\dist (\vx, Y_{\vj}^c) \]
where $Y_{\vj}^c=P_{d,n} \setminus Y_{\vj}$ is a closed set, and ``$\dist$'' denotes Euclidean distance. Note that for no $\vx \in P_{d,n}$ do all the entries of $S(\vx)$ vanish, because that would mean that $\vx \in Y_{\vj}^c$ for all $\vj \in [n]^d$, contradicting the fact that the sets $Y_{\vj}$ cover $P_{d,n}$. Note also that $S(\vx)$ is continuous in $\vx$.

Next, given the array $S(\vx)$ we denote by $\sigma^i_j(S(\vx))$, for $i=1,\ldots,d$, $j=1,\ldots,n$, the sum of the entries in the $j$-th layer in direction $i$, that is:
\[ \sigma^i_j(S(\vx))=\sum_{\vj \in [n]^d: j_i=j} s_{\vj}(\vx) \]
Then we define $A^i_j \subseteq P_{d,n}$ for $i=1,\ldots,d$, $j=1,\ldots,n$, by:
\[ \vx \in A^i_j \Leftrightarrow \sigma^i_j(S(\vx))=\max_{j' \in [n]} \sigma^i_{j'}(S(\vx)) \]
That is, $\vx \in A^i_j$ if the $j$-th layer in direction $i$ of $S(\vx)$ has the largest sum among the layers parallel to it.

We are going to apply Corollary \ref{pelegcor} to the sets $A^i_j$. By the continuity of $S(\vx)$, these sets are closed. For any direction $i$, one of the layers has the largest sum, hence $A^i_j$, $j=1,\ldots,n$, cover $P_{d,n}$. Finally, let us prove the implication $x^i_j=0 \Rightarrow \vx \notin A^i_j$. Assume, for contradiction, that $x^i_j=0$ and $\vx \in A^i_j$. As mentioned above, we deduce from $x^i_j=0$ that for any $\vj \in [n]^d$ with $j_i=j$ we have $\vx \notin Y_{\vj}$. Thus $\sigma^i_j(S(\vx))=0$, and we deduce from $\vx \in A^i_j$ that $\sigma^i_{j'}(S(\vx))=0$ for all $j' \in [n]$, meaning that all the entries of $S(\vx)$ vanish. As shown above, this is impossible.

By Corollary \ref{pelegcor}, we find a point $\vx \in \bigcap_{i=1}^d \bigcap_{j=1}^n A^i_j$. This means that in every direction, all layer-sums of $S(\vx)$ are equal. Hence, $\sigma^i_j(S(\vx))=a$ for some fixed $a$, regardless of $i$ and $j$. Clearly $a>0$, as $S(\vx)$ cannot entirely vanish.

We now construct an auxiliary $d$-partite hypergraph $\Gamma$. Its vertex set $V(\Gamma)$ is partitioned into $V^i=\{v^i_1,\ldots,v^i_n\}$, $i=1,\ldots,d$. We think of vertex $v^i_j$ as representing the $j$-th cell in the $i$-th $n$-partition corresponding to the point $\vx$ found in the previous paragraph. The edge set $E(\Gamma)$ consists of those sets $e_{\vj}=\{v^1_{j_1},\ldots,v^d_{j_d}\}$ such that $s_{\vj}(\vx)>0$. That is, an edge indexed by $\vj$ is present in $\Gamma$ if $\vx \in Y_{\vj}$, i.e., the $\vj$-th $d$-cell of $\vx$ contains an edge of $H$.

The function $f: E(\Gamma) \to \mathbb{R}_+$ defined by $f(e_{\vj})=\frac{s_{\vj}(\vx)}{a}$ is a fractional matching in $\Gamma$ satisfying $\sum_{e \in E(\Gamma)} f(e) = n$. It follows from Theorem \ref{furedi} that there is a matching $M$ in $\Gamma$ of size at least $\frac{n}{d-1}$. For every $e_{\vj} \in M$, the corresponding $\vj$-th $d$-cell of $\vx$ contains an edge of $H$, yielding the conclusion of Theorem \ref{disjointcells}. \hspace{\stretch{1}}$\square$

\medskip

{\em Proof of Theorem \ref{maintheorem}:} Given the $d$-interval hypergraph $H=(V,E)$ and the positive integer $n$, we define $Z_{\vj} \subseteq P_{d,n}$ for $\vj \in [n]^d$ by:
\[ \vx \in Z_{\vj} \,\,\Leftrightarrow \textrm{ the }\vj\textrm{-th }d\textrm{-cell of the }d \times n\textrm{-partition }\vx\textrm{ is contained in an edge of }H \]
The premise of the theorem says that the sets $Z_{\vj}$, $\vj \in [n]^d$, cover $P_{d,n}$. It is easy to check that the sets $Z_{\vj}$ are closed.

For a fixed $\e > 0$, we associate with every $\vx \in P_{d,n}$ a $d$-dimensional $n \times \cdots \times n$ array of nonnegative real numbers $T(\vx,\e)=(t_{\vj}(\vx,\e))_{\vj \in [n]^d}$, defined by:
\[ t_{\vj}(\vx,\e)=\max(1-\frac{\dist (\vx, Z_{\vj})}{\e}, 0) \]
For no $\vx \in P_{d,n}$ do all the entries of $T(\vx,\e)$ vanish, because that would mean that $\vx \notin Z_{\vj}$ for all $\vj \in [n]^d$, contradicting the fact that the sets $Z_{\vj}$ cover $P_{d,n}$. Note also that $T(\vx,\e)$ is continuous in $\vx$.

As in the previous proof, we consider for $i=1,\ldots,d$, $j=1,\ldots,n$, the sum of the entries in the $j$-th layer in direction $i$ of the array $T(\vx,\e)$, that is:
\[ \sigma^i_j(T(\vx,\e))=\sum_{\vj \in [n]^d: j_i=j} t_{\vj}(\vx,\e) \]
Then we define $B^i_j(\e) \subseteq P_{d,n}$ for $i=1,\ldots,d$, $j=1,\ldots,n$, by:
\[ \vx \in B^i_j(\e) \Leftrightarrow \sigma^i_j(T(\vx,\e))=\max_{j' \in [n]} \sigma^i_{j'}(T(\vx,\e)) \]

We are going to apply Theorem \ref{dualpeleg} to the sets $B^i_j(\e)$. The closedness and covering properties are verified just as in the previous proof. To check condition (b) of the theorem, we need the following:

\begin{claim}\label{claim}
Let $\vx \in P_{d,n}$ and $i \in  [d]$. Consider $\vj, \vk \in [n]^d$ such that $j_{i'}=k_{i'}$ for all $i' \in [d] \setminus \{i\}$, and $x^i_{j_i}=0$. Then $t_{\vj}(\vx,\e) \ge t_{\vk}(\vx,\e)$.
\end{claim}

\begin{proof} By the definition of $T(\vx,\e)$, it suffices to show that
\[ \dist (\vx, Z_{\vj}) \le \dist (\vx, Z_{\vk}). \]
Let $\vy \in Z_{\vk}$ be closest to $\vx$, and let $e$ be an edge of $H$ that contains the $\vk$-th $d$-cell of the $d \times n$-partition $\vy$ (by the definition of $Z_{\vk}$ such an edge exists). Consider $\vz \in P_{d,n}$ which coincides with $\vx$ in the $i$-th copy of $\Delta_{n-1}$, and with $\vy$ in all other copies. Then $\dist (\vx, \vz) \le \dist (\vx, \vy)$. Moreover, the $\vj$-th $d$-cell of $\vz$ consists of an empty cell in $U^i$ (because $z^i_{j_i}=x^i_{j_i}=0$), and otherwise coincides with the $\vk$-th $d$-cell of $\vy$, implying that the $\vj$-th $d$-cell of $\vz$ is also contained in $e$. This shows that $\vz \in Z_{\vj}$, and hence
\[ \dist (\vx, Z_{\vj}) \le \dist (\vx, \vz) \le \dist (\vx, \vy) = \dist (\vx, Z_{\vk}). \]
\end{proof}

Returning to the main proof, we verify the implication $x^i_j=0 \Rightarrow \vx \in B^i_j(\e)$. Indeed, assuming $x^i_j=0$, when we compare $\sigma^i_j(T(\vx,\e))$ and $\sigma^i_{j'}(T(\vx,\e))$, we find by Claim \ref{claim} that each term in the former is at least as large as the corresponding term in the latter. This implies that $\sigma^i_j(T(\vx,\e)) \ge \sigma^i_{j'}(T(\vx,\e))$ for all $j' \in [n]$, thus $\vx \in B^i_j(\e)$.

By Theorem \ref{dualpeleg}, we find a point $\vx(\e) \in \bigcap_{i=1}^d \bigcap_{j=1}^n B^i_j(\e)$. As before, this means that $\sigma^i_j(T(\vx(\e),\e))=a(\e)$ for some fixed $a(\e)>0$, regardless of $i$ and $j$.

We construct an auxiliary $d$-partite hypergraph $\Gamma(\e)$ in a way similar to the previous proof. The vertex set $V(\Gamma(\e))=V(\Gamma)$ is partitioned into $V^i=\{v^i_1,\ldots,v^i_n\}$, $i=1,\ldots,d$. The edge set $E(\Gamma(\e))$ consists of those sets $e_{\vj}=\{v^1_{j_1},\ldots,v^d_{j_d}\}$ such that $t_{\vj}(\vx(\e),\e)>0$. That is, an edge indexed by $\vj$ is present in $\Gamma(\e)$ if $\dist (\vx(\e), Z_{\vj}) < \e$.

The function $g: E(\Gamma(\e)) \to \mathbb{R}_+$ defined by $g(e_{\vj})=\frac{t_{\vj}(\vx(\e),\e)}{a(\e)}$ is a fractional edge-cover in $\Gamma(\e)$ satisfying $\sum_{e \in E(\Gamma(\e))} g(e) = n$. It follows from Corollary~\ref{lovaszcor} that there is an edge-cover $C(\e)$ in $\Gamma(\e)$ of size at most $(1+\ln d)n$.

Up till now, $\e>0$ was fixed. Because $P_{d,n}$ is compact and the number of possibilities for the edge-cover $C(\e)$ is finite, we can find a sequence of positive numbers $\e_k \to 0$ such that $\vx(\e_k)$ converges to some $\vx \in P_{d,n}$, and $C(\e_k)$ is a constant set $C$ of at most $(1+\ln d)n$ $d$-tuples that together cover $V(\Gamma)$. For any $e_{\vj} \in C$ we have $\dist (\vx(\e_k), Z_{\vj}) < \e_k$ for all $k$, implying that $\vx \in Z_{\vj}$. Thus, for any $e_{\vj} \in C$, the corresponding $\vj$-th $d$-cell of $\vx$ is contained in an edge of $H$. As the $e_{\vj} \in C$ cover $V(\Gamma)$, we have at most $(1+\ln d)n$ edges of $H$ that together cover all cells of $\vx$, and hence cover $V$, as required.

In the special case $d=2$, the hypergraphs $\Gamma(\e)$ are actually bipartite graphs, and so we may use Theorem \ref{bipgallai} instead of Corollary \ref{lovaszcor}, and get $n$ as an upper bound instead of $(1+\ln d)n$. \hspace{\stretch{1}}$\square$

\section{Are the bounds tight?}\label{examples}

The $\frac{n}{d-1}$ bound in Theorem \ref{disjointcells} is obviously tight for $d=2$. For large $d$, it has been shown by Matou\v{s}ek \cite{matousek} that the bound is tight up to a factor of order $\log ^2 d$.

Regarding Theorem \ref{maintheorem}, the following simple example shows that the bound $\rho(H) \le n$ for $d=2$ is best possible.

\begin{example}
Let $V=U^1 \cup U^2$, and let $n$ be a positive integer. Consider the $2$-interval hypergraph $H=(V,E)$ having $n^2$ edges of the form $e=I^1_k \cup I^2_\ell$, $k=1,\ldots,n$, $\ell=1,\ldots,n$, where $I^i_j = [\frac{j-1}{n},\frac{j}{n}] \subseteq U^i$. By the pigeonhole principle, every $2 \times n$-partition of $V$ has a $2$-cell that is contained in an edge of $H$. Clearly, $\rho(H)=n$. \hspace{\stretch{1}}$\lozenge$
\end{example}

\medskip

Moving to $d=3$, the following example shows that $\rho(H) \le n$ cannot be guaranteed anymore.

\begin{example}
Let $V=U^1 \cup U^2 \cup U^3$, and let $n=2$. For $i=1,2,3$, $v=0,1$, and $0 \le \ell < 1$, we denote by $I^i_{v,\ell}$ the closed subinterval of $U^i$ that contains its endpoint $v$ and has length $\ell$. We construct a $3$-interval hypergraph $H=(V,E)$ in which every edge is indexed by some $\vv \in \{0,1\}^3$ and some $\vl \in [0,1)^3$, and is of the form:
\[ e_{\vv,\vl} = I^1_{v_1,\ell_1} \cup I^2_{v_2,\ell_2} \cup I^3_{v_3,\ell_3} \]
We put $e_{\vv,\vl}$ in $E$ only for certain choices of $\vv, \vl$, namely those satisfying:
\begin{itemize}
\item $\ell_i \in \{ \frac{k}{24} \mid k=0,1,\ldots,23\}$, $i=1,2,3$
\end{itemize}
and either
\begin{itemize}
\item $\vv \in \{(0,0,0),(1,0,0),(0,1,0),(0,0,1)\}$ and $\ell_1 + \ell_2 + \ell_3 = \frac{47}{24}$
\end{itemize}
or
\begin{itemize}
\item $\vv \in \{(1,1,1),(0,1,1),(1,0,1),(1,1,0)\}$ and $\ell_1 + \ell_2 + \ell_3 = 1$
\end{itemize}

We proceed to show that the premise of Theorem \ref{maintheorem} is satisfied for $n=2$. Namely, that every $3 \times 2$-partition of $V$ has a $3$-cell that is contained in an edge of $H$. Let $$[0,c^1),(c^1,1];\,[0,c^2),(c^2,1];\,[0,c^3),(c^3,1]$$ be such a $3 \times 2$-partition. First, assume that $\frac{1}{24} \le c^i \le \frac{23}{24}$ for $i=1,2,3$. If $c^1 + c^2 +c^3 \le \frac{15}{8}$, then rounding each $c^i$ up to the nearest number of the form $\frac{k}{24}$, the sum will not exceed $\frac{47}{24}$, and hence the $3$-cell $[0,c^1) \cup [0,c^2) \cup [0,c^3)$ is contained in some $e_{\vv,\vl} \in E$ with $\vv =(0,0,0)$. Similarly, if $(1-c^1) + (1-c^2) + (1-c^3) \le \frac{7}{8}$, the $3$-cell $(c^1,1] \cup (c^2,1] \cup (c^3,1]$ is contained in some $e_{\vv,\vl} \in E$ with $\vv =(1,1,1)$. In either case we are done, therefore we may assume that:
\[ \frac{15}{8} < c^1 + c^2 + c^3 < \frac{17}{8} \]
Arguing similarly with respect to each of the other three pairs of antipodal $\vv \in \{0,1\}^3$, we see that we are done unless the following three constraints hold as well:
\[ \frac{15}{8} < (1-c^1) + c^2 + c^3 < \frac{17}{8} \]
\[ \frac{15}{8} < c^1 + (1-c^2) + c^3 < \frac{17}{8} \]
\[ \frac{15}{8} < c^1 + c^2 + (1-c^3) < \frac{17}{8} \]
Combining each of these three constraints with the first one above yields:
\[ |c^i - (1-c^i)| < \frac{1}{4}, \,\,i=1,2,3 \]
This, in turn, yields $\frac{3}{8} < c^i < \frac{5}{8}$, $i=1,2,3$, contradicting $c^1 + c^2 + c^3 > \frac{15}{8}$.

It remains to handle cases where one or more cells in the given $3 \times 2$-partition are shorter than $\frac{1}{24}$ (``short" for brevity). Obviously, there is at most one short cell in each $U^i$. If there are short cells in all $U^i$, then the $3$-cell that they form is easily contained in an edge of $H$. If there are short cells in two of the $U^i$, we look at the $3$-cell that they form together with the shorter of the two cells in the third $U^i$; again, it is easily contained in an edge of $H$. Finally, suppose that there is a short cell in only one of the $U^i$, say $U^1$. If the short cell is $[0,c^1)$, then the $3$-cell $[0,c^1) \cup [0,c^2) \cup [0,c^3)$ is contained in $e_{(0,0,0),(\frac{1}{24},\frac{23}{24},\frac{23}{24})}$. If the short cell is $(c^1,1]$ then the $3$-cell $(c^1,1] \cup [0,c^2) \cup [0,c^3)$ is contained in $e_{(1,0,0),(\frac{1}{24},\frac{23}{24},\frac{23}{24})}$. This completes the verification of the premise of Theorem \ref{maintheorem} for $n=2$.

On the other hand, $\rho(H)=3$ in this example. To see that there is no edge-cover of size $2$, assume to the contrary that $e_{\vv,\vl}, e_{\vec{v'},\vec{\ell'}} \in E$ cover $V$. Then necessarily $\vv,\vec{v'}$ are antipodal in $\{0,1\}^3$, and therefore, of the two sums $\ell_1 + \ell_2 + \ell_3$ and $\ell'_1 + \ell'_2 + \ell'_3$, one equals $\frac{47}{24}$ and the other $1$. Hence the total length of these two $3$-intervals is $\frac{71}{24}$, insufficient to cover $V$. \hspace{\stretch{1}}$\lozenge$
\end{example}

\medskip

The above example suggests that for $d=3$ it might be possible to improve the bound $\rho(H) \le (1+\ln 3)n$ to $\rho(H) \le \frac{3}{2}n$. While we are currently unable to achieve this, we can reduce the constant somewhat, and prove $\rho(H) < \frac{7}{4}n$. This follows from the respective improvement of the bound of Corollary \ref{lovaszcor} under the conditions relevant to our application.

To state this improvement, we recall that a function $f: E \to \mathbb{R}_+$ is called a {\em perfect fractional matching} (or equivalently a {\em perfect fractional edge-cover}) in $H=(V,E)$ if $\sum_{e \ni v} f(e) = 1$ for all $v \in V$.

\begin{proposition}\label{prop}
Let $H=(V,E)$ be a $3$-partite hypergraph, with $V=V^1 \cup V^2 \cup V^3$ where $|V^i|=n$ for $i=1,2,3$. If $H$ has a perfect fractional matching, then $\rho(H) < \frac{7}{4}n$.
\end{proposition}

\begin{proof} Let $f: E \to \mathbb{R}_+$ be a perfect fractional matching in $H$. Using only edges in the support of $f$, we construct an edge-cover $C$ in three steps:

Step 1: We find a matching $M$ of maximal size, and place its edges in $C$.

Step 2: For as long as we can find an edge covering two of the yet uncovered vertices, we add such edges to $C$.

Step 3: We cover the remaining uncovered vertices one-by-one, adding suitable edges to $C$.

By Theorem \ref{furedi}, we have $|M| \ge \frac{n}{2}$ for the matching in Step 1. We also observe that the set $V_0$ of uncovered vertices after the completion of Step 2 satisfies $|V_0| < n$. Indeed, the edges that intersect $V_0$ do so in only one vertex, and therefore contribute twice as much $f$-weight to $V \setminus V_0$ than to $V_0$. Since the total $f$-weight on every vertex is $1$, this implies that $2|V_0| \le |V \setminus V_0|$, or $|V_0| \le n$. The strict inequality comes from the fact that there is positive contribution of $f$-weight to $V \setminus V_0$ from edges used in the first two steps (recall that they are in the support of $f$).

Now we can estimate the number of edges used as follows:
\[ |C| = |M| + \frac{1}{2}(3n - 3|M| - |V_0|) + |V_0| = \frac{1}{2}(3n - |M| + |V_0|) < \frac{1}{2}(3n - \frac{n}{2} + n) = \frac{7}{4}n \]
\end{proof}

\begin{corollary}\label{sevenfourths}
Under the condition of Theorem \ref{maintheorem} for $d=3$, we have $\rho(H) < \frac{7}{4}n$.
\end{corollary}

An analogous improvement of the bound in Theorem \ref{maintheorem}, taking advantage of $d$-partiteness, can be obtained for arbitrary $d$. But it only leads to an insignificant improvement of the constant $1+\ln d$ for large $d$, hence we do not state it explicitly. Unfortunately, we do not know if the bound $\rho(H) \le (1+\ln d)n$ is close to best possible as $d$ grows to infinity.



\begin{thebibliography}{99}

\bibitem
{AKZ}
R. Aharoni, T. Kaiser, and S. Zerbib, Fractional covers and matchings in families of weighted $d$-intervals, to appear in Combinatorica, arXiv:1402.2064v2.

\bibitem{alon} N. Alon, Piercing $d$-intervals, Disc. Comput. Geom. 19 (1998), 333--334.


\bibitem{berger} E. Berger, KKM --- A topological approach for trees, Combinatorica 25 (2004),
1--18.


\bibitem{furedi}  Z. F\"{u}redi, Maximum degree and fractional matchings
in uniform hypergraphs, Combinatorica 1 (1981), 155--162.

\bibitem{gal59} T. Gallai, \"Uber extreme Punkt-und Kantenmengen (in German), Ann. Univ. Sci. Budapest, E\"otv\"os Sect. Math. 2 (1959), 133--138.

\bibitem{gal62} T. Gallai, Graphen mit triangulierbaren ungeraden Vielecken (in German), Magyar Tud. Ak. Mat. Kut. Int. K\"ozl. 7 (1962), 3--36.




\bibitem{gl2} A. Gy\'arf\'as and J. Lehel, Covering and coloring problems for relatives of intervals, Discrete Mathematics 55 (1985), 167--180.


\bibitem{kaiser} T. Kaiser, Transversals of
  $d$-intervals, Disc. Comput. Geom. 18 (1997), 195--203.

\bibitem{kkm} B. Knaster, C. Kuratowski, and S. Mazurkiewicz, Ein Beweis des Fixpunktsatzes f\"uŸr $n$-dimensionale Simplexe (in German), Fundamenta Mathematicae 14 (1929), 132--137.





\bibitem{lovasz} L. Lova\'sz, On the ratio of optimal integral and fractional covers, Discrete Math. 13 (1975), 383--390.

\bibitem{matousek} J. Matou\v{s}ek, Lower bounds on the transversal numbers of $d$-intervals,
Disc. Comput. Geom. 26 (2001), 283--287.


\bibitem{peleg} B. Peleg, Equilibrium points for open acyclic relations, Canad. J. Math. 19 (1967), 366--369.

\bibitem{peleg2} B. Peleg, Existence theorem for the bargaining set $M^{(i)}_1$, in: M. Shubik, ed., Essays in mathematical economics in honor of Oskar Morgenstern (Princeton University Press, Princeton, 1967), 53--56.




\bibitem{sperner} E. Sperner, Neuer Beweis f\"uŸr die Invarianz der Dimensionszahl und des Gebietes (in German), Abh. Math. Sem. Univ. Hamburg 6 (1928), 265--272.


\bibitem{tardos} G. Tardos, Transversals of 2-intervals, a topological approach,
Combinatorica 15 (1995), 123--134.

\end{thebibliography}
\end{document}